\documentclass[12pt,reqno]{amsart}

\usepackage[T1]{fontenc}
\usepackage[utf8]{inputenc}
\usepackage{lmodern}
\usepackage{microtype}
\usepackage{mathtools}
\usepackage{amssymb}
\usepackage{enumitem}
\usepackage{xcolor}
\usepackage{hyperref}
\usepackage[nameinlink,capitalise,noabbrev]{cleveref}

\hypersetup{
  colorlinks=true,
  linkcolor=blue!55!black,
  citecolor=green!45!black,
  urlcolor=blue!65!black,
  pdftitle={Module-Valued 2-Local Derivations on Reductive Lie Algebras},
  pdfauthor={Yang Chen, Yongqi Luo, Junzi Xu},
  pdfsubject={Lie algebras and representation theory},
  pdfkeywords={2-local derivation, reductive Lie algebra, finite-dimensional module,
    homogeneous map, nilpotent central action, invariant vector, affine realizer}
}

\newtheorem{theorem}{Theorem}[section]
\newtheorem{lemma}[theorem]{Lemma}
\newtheorem{proposition}[theorem]{Proposition}
\newtheorem{corollary}[theorem]{Corollary}
\theoremstyle{definition}

\theoremstyle{remark}
\newtheorem{remark}[theorem]{Remark}
\numberwithin{equation}{section}

\newcommand{\F}{\mathbb{F}}
\newcommand{\C}{\mathbb{C}}
\newcommand{\g}{\mathfrak{g}}
\newcommand{\s}{\mathfrak{s}}
\newcommand{\z}{\mathfrak{z}}
\newcommand{\h}{\mathfrak{h}}
\newcommand{\tori}{\mathfrak{h}}
\newcommand{\aLie}{\mathfrak{a}}

\newcommand{\Der}{\operatorname{Der}}
\newcommand{\Ad}{\operatorname{Ad}}

\newcommand{\Gr}{\operatorname{Gr}}
\newcommand{\Span}{\operatorname{span}}
\newcommand{\id}{\operatorname{id}}

\newcommand{\Homog}{\operatorname{Homog}}

\newcommand{\prz}{\operatorname{pr}_{\z}}

\begin{document}

\setcounter{page}{1}

\title[2-Local Derivations on Reductive Lie Algebras]
{Module-Valued 2-Local Derivations on Reductive Lie Algebras}

\author{Yang Chen}

\address{Y. Chen: School of Science, Jimei University, Xiamen, Fujian, 361021, P. R. China}

\email{chenyang1729@hotmail.com}

\author{Yongqi Luo}

\address{Y. Luo: School of Science, Jimei University, Xiamen, Fujian, 361021, P. R. China}

\email{luoyongqi777@outlook.com}

\author{Junzi Xu}

\address{J. Xu: School of Science, Jimei University, Xiamen, Fujian, 361021, P. R. China}

\email{2711681629@qq.com}

\date{}

\subjclass[2020]{17B40, 17B10, 17B20, 16W25}

\keywords{2-local derivation, reductive Lie algebra, module, homogeneous map}

\begin{abstract} Let \(\F\) be an algebraically closed field of characteristic zero,
\(\g=\s\oplus\z\) a finite-dimensional reductive Lie algebra over \(\F\),
and \(V\) an arbitrary finite-dimensional \(\g\)-module.  We classify all 2-local derivations of \(\g\) on \(V\),
and show that every 2-local derivation is a derivation if and only if
\(\dim\z\leq1\) or \(V^\g=0\).  If \(\dim\z\geq2\) and \(V^\g\ne0\), the
nonlinear homogeneous maps give all exceptional 2-local derivations.
\end{abstract}

\maketitle

\section{Introduction}

In 1997, \v{S}emrl \cite{Sem} introduced the notion of 2-local derivations on algebras,
whose values at every two points can be simultaneously realized by an actual derivation.
For Lie algebras, Ayupov, Kudaybergenov, and Rakhimov proved that every
2-local derivation of a finite-dimensional semisimple Lie algebra is a derivation \cite{AyuKudRak}.
Wang, Li, and Tang formulated the
module-valued problem and settled it for simple
\(\mathfrak{sl}_2(\C)\)-modules \cite{WLT2022}. The completely
reducible \(\mathfrak{sl}_2\)-case was subsequently obtained in
\cite{WTL2025}.  In this paper, we classify all 2-local derivations of the finite-dimensional reductive Lie algebra on
an arbitrary finite-dimensional module. At a conceptual level, this conclusion is a Lie-module analogue of the
nonlinear Gleason--Kahane--\.Zelazko principle of Kowalski and S{\l}odkowski
\cite{KS}: in both settings, two-point realizability forces
one global realizer. The main proof in this paper is inspired by \cite{YZ} and A. Premet's comments on \cite{YZ}.
Throughout, the ground field is the algebraically closed field \(\F\) of
characteristic zero, and all Lie algebras and modules are finite-dimensional.
We write \(xv\) for the action of \(x\in\g\) on \(v\in V\).

A derivation from the Lie algebra \(\g\) to a \(\g\)-module \(V\) is a linear map
\(D\colon\g\to V\) satisfying
\[
  D([x,y])=xD(y)-yD(x),
  \quad \forall x,y\in\g.
\]
A map \(\Delta\colon\g\to V\), with no linearity assumption, is a
\emph{2-local derivation} if for every \(x,y\in\g\) there exists a derivation
\(D_{x,y}\colon\g\to V\) such that
\[
  \Delta(x)=D_{x,y}(x),
  \qquad
  \Delta(y)=D_{x,y}(y).
\]
Let \(\operatorname{2LocDer}(\g,V)\) denote the set of all 2-local derivations from \(\g\) to \(V\).
The main results are as follows.

\begin{theorem}\label{thm:semisimple}
Let \(\g\) be a finite-dimensional semisimple Lie algebra over an
algebraically closed field \(\F\) of characteristic zero and let \(V\) be
a finite-dimensional \(\g\)-module.  Every 2-local derivation
\(\Delta\colon\g\to V\) is a derivation.
\end{theorem}

Write the reductive Lie algebra
\[
 \g=\s\oplus\z,\quad
 \s=[\g,\g],\quad \z=Z(\g),
\]
and let \(\prz\colon\g\to\z\) be the canonical projection.  For vector
spaces \(X,Y\), write
\[
 \Homog(X,Y)
 =\{H\colon X\to Y:H(\lambda x)=\lambda H(x)
   \text{ for all }\lambda\in\F,\ x\in X\}.
\]

\begin{theorem}\label{thm:reductive}
Let \(\g=\s\oplus\z\) be a finite-dimensional reductive Lie algebra over an
algebraically closed field \(\F\) of characteristic zero, and let \(V\) be a
finite-dimensional \(\g\)-module.  Then
\begin{equation}\label{eq:classification-intro}
 \operatorname{2LocDer}(\g,V)
 =
 \Der(\g,V)
 +
 \{H\circ\prz:H\in\Homog(\z,V^\g)\}.
\end{equation}
Moreover,
\begin{equation}\label{eq:intersection-intro}
 \Der(\g,V)\cap\{H\circ\prz:H\in\Homog(\z,V^\g)\}
=\{L\circ\prz:L\in\operatorname{Hom}_{\F}(\z,V^\g)\}.
\end{equation}
Equivalently, every such map admits the
unconditional formula
\[
 \Delta(s+z)= (s+z)u+f(z)+H(z),
\]
\[
\quad u\in \s V, \quad f\in \Der(\z,V^\s), \quad H\in\Homog(\z,V^\g).
\]
In particular,
\begin{equation}\label{eq:quotient-classification}
 \frac{\operatorname{2LocDer}(\g,V)}{\Der(\g,V)}
 \cong
 \frac{\Homog(\z,V^\g)}
 {\operatorname{Hom}_{\F}(\z,V^\g)}.
\end{equation}
\end{theorem}

\begin{corollary}\label{cor:criterion}
Every 2-local derivation
\(\Delta: \g\to V\) is a derivation if and only if
\[
 \dim\z\leq1
 \qquad\text{or}\qquad
 V^\g=0.
\]
If \(V\) is completely reducible, every 2-local derivation has the form
\begin{equation}\label{eq:cr-form-intro}
 \Delta(s+z)=(s+z)v+H(z),
 \quad
 v\in V,\quad H\in\Homog(\z,V^\g).
\end{equation}
\end{corollary}

Let $V$ be a finite-dimensional $\mathfrak g$-module. For \(v\in V\), define
\[
  D_v\colon\g\longrightarrow V,
  \quad x\longmapsto xv.
\]
We shall repeatedly use the elementary fact that
\(\Delta-D_v\) is 2-local whenever \(\Delta\) is 2-local.
The following lemma is easy to understand.

\begin{lemma}\label{lem:direct-sum-reduction}
Suppose that \(V=\bigoplus_{i=1}^rV_i\) is a direct sum of
\(\g\)-submodules, and let \(\pi_i\colon V\to V_i\) be the projections.  If
\(\Delta\colon\g\to V\) is 2-local, then
\(\Delta_i=\pi_i\circ\Delta\colon\g\to V_i\) is 2-local for every \(i\).  If every
\(\Delta_i\) is a derivation, then so is \(\Delta\).
\end{lemma}

\section{Affine realizers and reductions}

Throughout Sections~2--4, $\mathfrak g$ denotes a
finite-dimensional semisimple Lie algebra over $\F$.
We return to the general reductive decomposition
$\mathfrak g=\mathfrak s\oplus\mathfrak z$ in Section~6.

The first Whitehead lemma gives \(H^1(\g,V)=0\), and therefore
\[
  \Der(\g,V)=\{D_v:v\in V\},
\]
where \(v\) is unique modulo \(V^\g\).
For \(x\in\g\), put
\[
  K_x=\ker(x|_V)
  \quad\text{and}\quad
  \mathcal R_x(\Delta)=\{v\in V:xv=\Delta(x)\}.
\]
The map \(\Delta\colon\g\to V\) is a 2-local
derivation if and only if
\[
  \mathcal R_x(\Delta)\cap\mathcal R_y(\Delta)\ne\varnothing, \quad \forall x,y\in\g.
\]
Whenever these conditions hold, each \(\mathcal R_x(\Delta)\) is a nonempty
affine space of the form
\[
  \mathcal R_x(\Delta)=v_x+K_x.
\]
Moreover, \(\Delta=D_v\) for some \(v\in V\) if and only if
\[
  v\in \bigcap_{x\in\g}\mathcal R_x(\Delta)\ne\varnothing.
\]

\begin{lemma}\label{lem:two-anchor}
Let \(\Delta\colon\g\to V\) be a 2-local derivation and suppose that
\(\Delta(a)=\Delta(b)=0\).  Then
\[
  \Delta(x)\in xK_a\cap xK_b, \quad x\in\g.
\]
In particular, if \(K_a+ K_b+ K_x\) is a direct sum, then \(\Delta(x)=0\).
\end{lemma}

\begin{proof}
Apply the 2-local property to \((a,x)\).  There is \(u\in V\) such that
\[
  0=\Delta(a)=au,
  \qquad
  \Delta(x)=xu.
\]
Thus \(u\in K_a\).  Similarly, there is \(w\in K_b\) with
\(\Delta(x)=xw\), proving 
\[
  \Delta(x)\in xK_a\cap xK_b, \quad x\in\g.
\]
If \(K_a+ K_b+ K_x\) is a direct sum, then \(u-w\in K_x\), \(u\in K_a\) and \(-w\in K_b\),
which forces \(u=w=0\), and hence \(\Delta(x)=0\).
\end{proof}

We record the standard reductions needed later.

\begin{lemma}\label{lem:faithful-quotient}
Let \(V\) be an irreducible nontrivial \(\g\)-module, and
\[
  \operatorname{Ann}_\g V=\{x\in\g:xV=0\}.
\]
Every 2-local derivation \(\Delta\colon\g\to V\) factors through a 2-local
derivation
\[
  \overline\Delta\colon\g/\operatorname{Ann}_\g V\longrightarrow V,
\]
where the quotient \(\g/\operatorname{Ann}_\g V\) is semisimple and acts faithfully on \(V\).
\end{lemma}

\begin{proof}
Every derivation from \(\g\) to \(V\) is of the form \(D_v\) and therefore vanishes on
\(\operatorname{Ann}_\g V\).  For \(x_0\in\operatorname{Ann}_\g V\), applying 2-locality to
\((x,x+x_0)\), we have \(\Delta(x+x_0)=\Delta(x)\).  Thus
\(\overline\Delta(x+\operatorname{Ann}_\g V)=\Delta(x)\) is well defined.  A derivation
realizing two values of \(\Delta\) factors through \(\g/\operatorname{Ann}_\g V\), which
proves 2-locality of \(\overline\Delta\).  Since \(\operatorname{Ann}_\g V\) is an ideal
of the semisimple Lie algebra \(\g\), the quotient is semisimple.
\end{proof}

By complete reducibility and \cref{lem:direct-sum-reduction,lem:faithful-quotient}, it is enough to prove
\cref{thm:semisimple} when \(V\) is irreducible, nontrivial, and faithful.
These assumptions remain in force through \cref{sec:proof-of-thm1.1}.

Let \(G\) be a simply connected semisimple algebraic
group with Lie algebra \(\g\).  The representation of \(\g\) on \(V\)
integrates uniquely to a rational representation of \(G\).
Fix a Cartan subalgebra \(\tori\subseteq\g\), and write
\[
  V=\bigoplus_{\mu\in\tori^*}V_\mu.
\]
Choose a system \(\Pi\) of simple roots and put
\[
 Q_+=\sum_{\alpha\in\Pi}\mathbb Z_{\geq0}\alpha,
 \qquad V_+=\bigoplus_{\mu\in Q_+\setminus\{0\}}V_\mu,
 \qquad V_-=\bigoplus_{\mu\in -Q_+\setminus\{0\}}V_\mu.
\]
Choose nonzero root vectors
\[
  e_\alpha\in\g_\alpha,\qquad
  f_\alpha\in\g_{-\alpha},
  \qquad \alpha\in\Pi.
\]
Set
\[
  E=\sum_{\alpha\in\Pi}e_\alpha,
  \qquad
  F=\sum_{\alpha\in\Pi}f_\alpha,
  \qquad
  p=\exp_G(E),
  \qquad
  q=\exp_G(F).
\]
An element \(h\in\tori\) is called \emph{\(V\)-regular relative to
\(\tori\)} if
\[
  K_h=\ker(h|_V)= V_0.
\]
These elements form a nonempty Zariski-open dense subset of \(\tori\).

\begin{lemma}\label{lem:three-spaces}
The sum \(V_0+ pV_0+ qV_0\) is direct.
\end{lemma}

\begin{proof}
Suppose
\[
  v_0+pv_++qv_-=0,
  \quad v_0,v_+,v_-\in V_0,
\]
then
\[
  V_+\ni(p-\id)v_+= 0= (q-\id)v_-\in V_-.
\]
Since \(E\) is nilpotent, we have
\[
  v_+\in \ker((p-\id)|_{V_0})=\ker E|_{V_0},
\]
which implies that \(e_\alpha v_+=0\), \(\forall\alpha\in \Pi\).
Irreducibility of the nontrivial module $V$ gives
\(v_+=0\). Similarly, \(v_-=0\), and then \(v_0=0\).
\end{proof}

\begin{remark}\label{rem:open-transversality}
For any \(V\)-regular \(h_0\in\tori\), equivariance gives
\[
  K_{\Ad(g)h_0}=gK_{h_0}=gV_0, \quad g\in G.
\]
Consider the orbit morphism
\[
  \kappa\colon G\longrightarrow\Gr(\dim V_0,V),
  \qquad g\longmapsto gV_0.
\]
Since
\[
  \Omega=\{W\in\Gr(\dim V_0,V):W\cap (V_0\oplus pV_0)=0\}
\]
is Zariski open, so is
\[
  U:=\{g\in G:V_0+pV_0+gV_0\text{ is direct}\}
    =\kappa^{-1}(\Omega).
\]
Moreover, \(U\) is nonempty and dense in the irreducible variety \(G\).
\end{remark}

\section{The semisimple part}

\begin{proposition}\label{prop:open-realization}
Let \(h_0\in\tori\) be \(V\)-regular relative to \(\tori\), and
\(\Delta\colon\g\to V\) a 2-local derivation.  There exist \(v_{h_0}\in V\)
and a nonempty Zariski-open subset \(\Omega_{h_0}\subseteq\g\) such that
\(\Delta(x)=xv_{h_0}\), \(x\in \{h_0\}\cup\Omega_{h_0}\).
\end{proposition}

\begin{proof}
Applying 2-locality to the pair \((h_0, \Ad(p)h_0)\), there is
\(v_{h_0}\in V\) such that
\[
  \Delta(h_0)=h_0v_{h_0},\qquad
  \Delta(\Ad(p)h_0)=(\Ad(p)h_0)v_{h_0}.
\]
Set \(\Delta'=\Delta-D_{v_{h_0}}\), then
\[
\Delta'(h_0)=\Delta'(\Ad(p)h_0)=0.
\]
For \(g\in U\), then
\[
  K_{h_0}+ K_{\Ad(p)h_0}+ K_{\Ad(g)h_0}
  =V_0+ pV_0+ gV_0
\]
is direct, which implies that \(\Delta'(\Ad(g)h_0)=0\) by \Cref{lem:two-anchor}.
For \(h\in\tori\), a derivation realizing \(\Delta'\) at
\(\Ad(g)h_0\) and \(\Ad(g)h\) has the form \(D_{v_h}\). Then
\[
  v_h\in K_{\Ad(g)h_0}=gV_0.
\]
Since \(hV_0=0\), it follows that
\[
  \Delta'(\Ad(g)h)=0, \quad g\in U,\quad h\in\tori.
\]

Consider the adjoint-action morphism
\[
\Psi\colon G\times\tori\longrightarrow\g, \qquad (g,h)\longmapsto\Ad(g)h.
\]
The image of \(\Psi\) contains the regular semisimple locus, which is nonempty and Zariski open in \(\g\). Therefore the morphism \(\Psi\) is dominant, and its restriction to the dense open subset \(U\times\tori\) is dominant. By Chevalley's theorem, \(\Psi(U\times\tori)\) contains a nonempty Zariski-open subset \(\Omega_{h_0}\subseteq\g\) as required.
\end{proof}

\begin{proposition}\label{prop:all-semisimple}
For \(\Delta\in \operatorname{2LocDer}(\g,V)\), there exists \(v\in V\) such that \(\Delta(s)=sv\) for every semisimple element \(s\in \g\).
\end{proposition}

\begin{proof}
Let \(h_0\) and \(h'_0\) be \(V\)-regular elements relative to \(\h\) and \(\h'\) respectively.
Applying \cref{prop:open-realization} to \((h_0, \h)\) and \((h'_0, \h')\), we obtain
vectors \(v_{h_0},v_{h'_0}\) and nonempty open subsets
\(\Omega_{h_0},\Omega_{h'_0}\subseteq\g\). The irreducibility of the affine space \(\g\) gives
\[
  \Omega_{h_0}\cap\Omega_{h'_0}\ne\varnothing.
\]
It implies that the linear map
\[
  \g\longrightarrow V,\qquad y\longmapsto y(v_{h_0}-v_{h'_0})
\]
vanishes on a nonempty open set, and therefore vanishes identically. 
Thus
\[
  v_{h_0}-v_{h'_0}\in V^\g=0.
\]
Consequently all the vectors \(v_{h_0}\) coincide, and denote their common
value by \(v\).

Put \(\Delta'= \Delta- D_v\). Let \(s\) be semisimple and contained in a Cartan subalgebra
\(\tori\).  Choose a \(V\)-regular element \(h_0\in\tori\).
Applying 2-locality to \((s,h_0)\), there is \(v'\in V\) such that
\[
  \Delta'(s)=sv',\qquad
  0=\Delta'(h_0)=h_0 v',
\]
which force \(v'\in K_{h_0}=V^\tori\) and \(\Delta'(s)=0\). Hence \(\Delta(s)= D_v(s)= sv\).
\end{proof}

\section{Proof of Theorem~\ref{thm:semisimple}}\label{sec:proof-of-thm1.1}

\begin{lemma}\label{lem:sl2-nilpotent-separation}
Let $e,h,f$ be the standard generators of $\mathfrak{sl}_2$,
and $V$ a finite-dimensional $\mathfrak{sl}_2$-module.
If $S\in\operatorname{End}_{\mathfrak{sl}_2}(V)$ is semisimple,
then there exist semisimple elements $q_0,q_1\in\mathfrak{sl}_2$,
such that, for every operator $T\in\operatorname{End}_{\mathfrak{sl}_2}(V)$ with $[S,T]=0$,
\[
(S+e+T)\ker(S+q_0)\cap(S+e+T)\ker(S+q_1)\cap(S+e)\ker T= 0.
\]
In particular, when $T=0$, it follows that
\[
 (S+e)\ker(S+q_0)\cap(S+e)\ker(S+q_1)=0.
\]
\end{lemma}

\begin{proof}
Since \(S\) acts semisimply and commutes with \(\mathfrak{sl}_2\), put
\[
  V_\lambda=\ker(S-\lambda\id_V),
  \qquad
  M_{\lambda,m}=\operatorname{Hom}_{\mathfrak{sl}_2}\bigl(L(m),V_\lambda\bigr),
\]
where \(L(m)\) denotes the simple \(\mathfrak{sl}_2\)-module of highest
weight \(m\).  The canonical evaluation maps give the joint isotypic
decomposition
\[
  V\cong
  \bigoplus_{\lambda,m}M_{\lambda,m}\otimes L(m),
\]
where \(\mathfrak{sl}_2\) acts on the second factor and \(S\) acts by the scalar
\(\lambda\).

Choose \(t\in\F^\times\) such that
\[
  \lambda+tj\ne0
\]
whenever \(M_{\lambda,m}\ne0\) and \(j\ne0\) is an \(h\)-weight of
\(L(m)\).  If \(v_{k,0}\) is a nonzero zero-weight vector in \(L(2k)\),
then
\[
  \ker(S+th)=\bigoplus_{k\ge0} M_{0,2k}\otimes\F v_{k,0}.
\]
Set
\[
  q_0=th,\qquad
  q_1=t\Ad(\exp_{SL_2}(f))h= t(h+2f).
\]
Both \(q_0\) and \(q_1\) are semisimple, and
\[
  \ker(S+q_1)=\exp_{SL_2}(f)\ker(S+q_0).
\]

It follows that
\[
(S+e+T)\ker(S+q_0)\cap (S+e+T)\ker(S+q_1)\subseteq \bigoplus_{k\ge0} M_{0,2k}\otimes L(2k).
\]
For $k=0$, both $S$ and $e$ act as zero on
$M_{0,0}\otimes L(0)$, so the corresponding component of
$(S+e)\ker T$ is zero. If $y$ lies in the threefold intersection on
$M_{0,2k}\otimes L(2k)$, $k\ge 1$, then
\[
\begin{aligned}
 y&=Tm_0\otimes v_{k,0}+m_0\otimes ev_{k,0}\\
  &=Tm_1\otimes \exp_{SL_2}(f)v_{k,0}+m_1\otimes e\exp_{SL_2}(f)v_{k,0}, \quad m_0,m_1\in M_{0,2k},
\end{aligned}
\]
and
\[
 y\in (S+e)\ker T\subseteq\ker T.
\]
We have
\[
Ty= T^2m_0\otimes v_{k,0}+ Tm_0\otimes ev_{k,0}= 0.
\]
The linear independence of $v_{k,0}$ and $ev_{k,0}$ gives \(Tm_0=0\).
Similarly, \(Tm_1=0\). Therefore
\[
 y=m_0\otimes ev_{k,0}
  =m_1\otimes e\exp_{SL_2}(f)v_{k,0},
\]
which forces $y=0$ by the linear independence of $ev_{k,0}$ and $e\exp_{SL_2}(f)v_{k,0}$.
Thus the threefold intersection is zero on every joint summand on
which the first two spaces are supported, and it is therefore zero on $V$.

Finally, when $T=0$, one has $\ker T=V$, and the first two spaces
are contained in $(S+e)V$. Hence the threefold intersection reduces
to
\[
 (S+e)\ker(S+q_0)\cap(S+e)\ker(S+q_1),
\]
which proves the last assertion.
\end{proof}

\begin{proof}[Proof of \cref{thm:semisimple}]
Assume \(V\) is irreducible, nontrivial, and faithful. By \cref{prop:all-semisimple}, there exists \(v\in V\) such that
\(\Delta'=\Delta-D_{v}\) vanishes on every semisimple element in \(\g\).
For a nonsemisimple element \(x\in\g\), write its Jordan
decomposition as
\[
  x=s+e,\qquad [s,e]=0,
\]
where \(s\) is semisimple and \(e\ne 0\) is nilpotent.
The centralizer \(\g^s\) is reductive, and \(e\) lies in its derived algebra.
The Jacobson--Morozov theorem supplies a triple \((e,h,f)\subseteq\g^s\).
Set
\[
  \mathfrak l_e= \Span_\F\{e,h,f\}\cong\mathfrak{sl}_2.
\]
Choose \(q_0,q_1\) as in \cref{lem:sl2-nilpotent-separation}.
The elements \(s+q_0\) and \(s+q_1\) are semisimple, so
\[
  \Delta'(s+q_0)= \Delta'(s+q_1)=0.
\]
For each \(j\), apply 2-locality to the pair \((x,s+q_j)\). Consequently,
\[
  \Delta'(x)\in(s+e)K_{s+q_0}\cap (s+e)K_{s+q_1},
\]
which forces \(\Delta'(x)=0\). Thus \(\Delta=D_v\) is a derivation.
\end{proof}

\section{Module-valued 2-local derivations on abelian Lie algebras}

The following elementary lemma isolates all possible nonlinear behavior in
the classification. It applies to an arbitrary module over an abelian Lie
algebra.

\begin{lemma}\label{lem:abelian-classification}
Let \(\aLie\) be a finite-dimensional abelian Lie algebra and \(M\) a \(\aLie\)-module. A map
\(\delta\colon\aLie\to M\) is a 2-local derivation if and only if
\[
 \delta=f+H,
 \quad
 f\in\Der(\aLie,M),\quad
 H\in\Homog(\aLie,M^\aLie),
\]
where
\[
\Der(\aLie,M)\cap \Homog(\aLie,M^\aLie)= \operatorname{Hom}_{\F}(\aLie,M^\aLie).
\]
\end{lemma}

\begin{proof}
Two-locality first gives homogeneity \(\delta(\lambda a)=\lambda\delta(a)\).
For \(a,b\in\aLie\), we have \(a\delta(b)=b\delta(a)\) by a derivation \(D_{a,b}\).
For \(a,b,c\in\aLie\), it follows that
\[
 c\bigl(\delta(a+b)-\delta(a)-\delta(b)\bigr)=(a+b)\delta(c)-a\delta(c)-b\delta(c)=0,
\]
which implies that the induced map \(\overline\delta\colon\aLie\longrightarrow M/M^\aLie\) is linear.

Choose a linear lift \(f\colon\aLie\to M\) of \(\overline\delta\), and put
\(H=\delta-f\).  Then \(H(\aLie)\subseteq M^\aLie\), and \(H\) is
homogeneous.  Since elements of \(\aLie\) annihilate \(H(\aLie)\), we have
\[
 af(b)=a\delta(b)=b\delta(a)=bf(a),
\]
which implies that \(f\in\Der(\aLie,M)\).

Conversely, let \(\delta=f+H\) have the displayed form.  For any \(a,b\in \aLie\),
homogeneity allows one to choose a linear map
\(L_{a,b}\colon\aLie\to M^\aLie\) satisfying
\[
 L_{a,b}(a)=H(a),\qquad L_{a,b}(b)=H(b).
\]
Then \(f+L_{a,b}\) is a derivation realizing \(\delta\) at \(a,b\).
\end{proof}

\section{Proof of Theorem~\ref{thm:reductive}}

We now return to a reductive Lie algebra
\[
 \g=\s\oplus\z,\qquad \s=[\g,\g],\quad \z=Z(\g),
\]
then \(V= \s V\oplus V^\s\) as a \(\g\)-module, and
\[
 \Der(\g,V)=\{D_u:u\in \s V\}\oplus \Der(\z,V^\s).
\]

\begin{lemma}\label{lem:core-rigidity}
Every 2-local derivation \(\Delta_{\s V}\colon\g\to \s V\) is an inner derivation.
\end{lemma}

\begin{proof}
All derivations from \(\g\) to \(\s V\) are inner.
The restriction of \(\Delta_{\s V}\) to \(\s\) is a 2-local derivation.  By
\cref{thm:semisimple}, let \(\Delta'= \Delta_{\s V}- D_v\) for some \(v\in \s V\) such that
\(\Delta'(s)= 0\), \(\forall s\in \s\).
For \(z\in\z\) and \(\forall s\in\s\), two-locality at \((s,z)\) gives \(\Delta'(z)\in zK_s\subseteq K_s\),
which implies \(\Delta'(z)\in \bigcap_{s\in\s}K_s=(\s V)^\s=0\). Hence \(\Delta'(z)=0\).

The $\z$-module $\s V$ has a generalized weight space decomposition
\[
 \s V=\bigoplus_{\chi\in\z^*}(\s V)^{(\chi)},
\]
where \((\s V)^{(\chi)}\) is a \(\g\)-submodule.
It is therefore enough to treat
\[
\Delta'_{\chi}=\pi_{\chi}\circ\Delta'\colon\g\to (\s V)^{(\chi)}.
\]
For \(\chi\ne0\), choose \(z_0\in\z\) with
\(\chi(z_0)\ne0\), then \(z_0\) acts invertibly on \((\s V)^{(\chi)}\).
Two-locality at \((z_0,x)\), together with \(\Delta'_{\chi}(z_0)=0\), gives
\(\Delta'_{\chi}(x)=0\) for every \(x\in\g\).
For \(\chi=0\),  every \(z\in\z\) acts nilpotently.
Fix \(x=a+z\), where \(a\in\s\) and \(z\in\z\).  If \(a\) is semisimple, two-locality at \((a,x)\) and
\((z,x)\) gives
\[
 \Delta'_{\chi}(x)\in z K_a\cap aK_z\subseteq K_a\cap \operatorname{im}a= 0,
\]
which forces \(\Delta'_{\chi}(x)=0\).
If \(a\) has the Jordan decomposition in \(\s\)
\[
 a=s+e,\quad [s,e]=0,
\]
with \(s\) semisimple and \(e\ne 0\) nilpotent, choose a Jacobson--Morozov triple \((e,h,f)\subseteq\s^s\)
and put \(\mathfrak l_e=\Span\{e,h,f\}\).  Apply
\cref{lem:sl2-nilpotent-separation} to the actions of \(s\), \(z\) and
\(\mathfrak l_e\).  It gives semisimple \(q_0,q_1\in\mathfrak l_e\) satisfying
\[
(s+e+z)\ker(s+q_0)\cap(s+e+z)\ker(s+q_1)\cap(s+e)\ker z= 0.
\]  
Two-locality at \((s+q_i,x)\) implies
\[
 \Delta'_{\chi}(x)\in(s+e+z)K_{s+q_i},\quad i=0,1.
\]
Two-locality at \((z,x)\) also gives
\[
 \Delta'_{\chi}(x)\in(s+e)K_z.
\]
Thus \(\Delta'_{\chi}(x)=0\).
We have proved \(\Delta'=0\) on every generalized central-weight summand. Hence \(\Delta_{\s V}= D_v\).
\end{proof}

\begin{proof}[Proof of \cref{thm:reductive}]
For \(\Delta\in \operatorname{2LocDer}(\g,V)\), let
\[
\Delta_{\s V}= \operatorname{pr}_{\s V}\circ\Delta, \qquad\Delta_{V^\s}= \operatorname{pr}_{V^\s}\circ\Delta,
\]
where \(\operatorname{pr}_{\s V}, \operatorname{pr}_{V^\s}\) are the canonical projections.
Then \(\Delta_{\s V}\) and \(\Delta_{V^\s}\) are both 2-local derivations.
By \cref{lem:core-rigidity}, there is \(u\in \s V\) such that
\(\Delta_{\s V}=D_u\).  Every derivation from \(\g\) to \(V^\s\) vanishes on \(\s\), so
\(\Delta_{V^\s}\) depends only on \(\z\), and induces a 2-local derivation
\[
 \delta_{V^\s}\colon\z\longrightarrow V^\s.
\]
By \cref{lem:abelian-classification},
\[
 \delta_{V^\s}=f+H,\qquad
 f\in\Der(\z,{V^\s}),\quad
 H\in\Homog(\z,V^\g).
\]
The map \(D=D_u+f\) is a derivation, and \(\Delta=D+H\circ\prz\).

Conversely, take \(D\in\Der(\g,V)\) and
\(H\in\Homog(\z,V^\g)\).  Given \(x,y\in\g\), choose a linear map
\(L_{x,y}\colon\z\to V^\g\) interpolating \(H\) at
\(\prz(x),\prz(y)\).  Then \(D+L_{x,y}\circ\prz\)
is a derivation and realizes \(D+H\circ\prz\) at \(x,y\).  This proves
\eqref{eq:classification-intro}.  Finally, \(H\circ\prz\) is a derivation
exactly when it is linear, proving \eqref{eq:intersection-intro}.
\end{proof}

\begin{proof}[Proof of \cref{cor:criterion}]
If \(V^\g=0\), the homogeneous summand in
\eqref{eq:classification-intro} is zero.  If \(\dim\z\leq1\), every
homogeneous map from \(\z\) to \(V^\g\) is linear.  In both cases all 2-local
derivations are derivations.
Conversely, suppose that \(\dim\z\ge2\) and \(V^\g\ne0\).  Choose linearly
independent functionals \(\alpha,\beta\in\z^*\) and \(0\ne v_0\in V^\g\).
The map
\[
 H(z)=
 \begin{cases}
   \alpha(z)^2\beta(z)^{-1}v_0,&\beta(z)\ne0,\\
   0,&\beta(z)=0,
 \end{cases}
\]
is homogeneous and nonadditive.  Hence \(H\circ\prz\) is a 2-local
derivation which is not a derivation.

If \(V\) is completely reducible, its nontrivial simple summands have
vanishing first cohomology, while the trivial summand is \(V^\g\).  Thus
every derivation is of the form
\[
 D(s+z)=(s+z)v+L(z),
 \qquad L\in\operatorname{Hom}_{\F}(\z,V^\g).
\]
Absorbing \(L\) into \(H\) gives \eqref{eq:cr-form-intro}.
\end{proof}

\vspace{2mm}
\noindent
{\bf Acknowledgments. } This research is partially supported by the National Natural Science Foundation of China (No. 12271210) and the Natural Science Foundation of Fujian Province, China (No. 2021J01862).


\begin{thebibliography}{22}

\bibitem{AyuKudRak} Sh. A. Ayupov,  K.  K.  Kudaybergenov,  I. S. Rakhimov, 2-Local derivations on finite-dimensional Lie algebras, {\it Linear Algebra and its Applications}, \textbf{474}, (2015), 1-11.

\bibitem{KS} S. Kowalski, Z. S{\l}odkowski, A characterization of multiplicative linear functionals in {B}anach algebras, {\it Studia Mathematica}, \textbf{67}, (1980), 215--223.

\bibitem{Sem} P. \v{S}emrl,  Local automorphisms and derivations on $B(H)$, {\it Proc. Amer. Math. Soc.}, \textbf{125}, (1997), 2677--2680.

\bibitem{WLT2022} S. Wang, Z. Li, X. Tang, 2-local derivations and biderivations of $\mathfrak{sl}(2)$ on all simple modules, {\it arXiv:2206.07974}, 2022.

\bibitem{WTL2025} S. Wang, X. Tang, Z. Li, Local and 2-local derivations of {$\mathfrak{sl}(2)$} on any finite-dimensional completely reducible module, {\it Journal of Algebra and Its Applications}, \textbf{24}, (2025), 2550197.
    
\bibitem{YZ} Y. Yao, K. Zhao, Local properties of {J}acobson--{W}itt algebras, {\it Journal of Algebra}, \textbf{586}, (2021), 1110--1121.

\end{thebibliography}
\end{document}